\documentclass[12pt,reqno]{amsart}

\usepackage[latin1]{inputenc}
\usepackage{amsmath}
\usepackage{amsfonts}
\usepackage{amssymb}
\usepackage{graphics}
\usepackage{color}
\usepackage{enumerate}
\usepackage{graphicx}
\usepackage{amscd}
\usepackage{amsthm}
\usepackage{latexsym}
\usepackage{verbatim}
\usepackage{hyperref}

\theoremstyle{plain}
\newtheorem{theorem}{Theorem}[section]
\newtheorem{thm}[theorem]{Theorem}

\newtheorem{prop}[theorem]{Proposition}

\theoremstyle{definition}

\newtheorem{conj}[theorem]{Conjecture}
\newtheorem{defn}[theorem]{Definition}
\newtheorem{ex}[theorem]{Example}
\newtheorem{OQ}{Open Question}

\theoremstyle{remark}
\newtheorem{rem}[theorem]{Remark}

\newcommand{\R}{\mathbb{R}}
\newcommand{\C}{\mathbb{C}}

\newcommand{\T}{\mathbb{T}}
\newcommand{\D}{\mathbb{D}}
\newcommand{\N}{\mathbb{N}}

\newcommand{\Pn}{\mathcal{P}_n}

\DeclareMathOperator{\sgn}{sgn}

\DeclareMathOperator*{\esssup}{ess\,sup}

\allowdisplaybreaks[3]


\begin{document}
\title[OPAs in $L^p$]{Optimal Polynomial Approximants in $L^p$}
\author[Centner]{Raymond Centner}
\address{Department of Mathematics and Statistics, University of South Florida, 4202 E. Fowler Avenue,
Tampa, Florida 33620-5700, USA.} \email{rcentner@usf.edu}

%\date{\today}

\begin{abstract}
Over the past several years, optimal polynomial approximants (OPAs) have been studied in many different function spaces.  In these settings, numerous papers have been devoted to studying the properties of their zeros.  In this paper, we introduce the notion of optimal polynomial approximant in the space $L^p$, $1\leq p\leq\infty$. We begin our treatment by showing existence and uniqueness for $1<p<\infty$.  For the extreme cases of $p=1$ and $p=\infty$, we show that uniqueness does not necessarily hold.  We continue our development by elaborating on the special case of $L^2$.  Here, we create a test to determine whether or not a given 1st degree OPA is zero-free in $\overline{\D}$.  Afterward, we shed light on an orthogonality condition in $L^p$.  This allows us to study OPAs in $L^p$ with the additional tools from the $L^2$ setting.  Throughout this paper, we focus many of our discussions on the zeros of OPAs.  In particular, we show that if $1<p<\infty$, $f\in H^p$, and $f(0)\neq 0$, then there exists a disk, centered at the origin, in which all the associated OPAs are zero-free.  Toward the end of this paper, we use the orthogonality condition to compute the coefficients of some OPAs in $L^p$.  To inspire further research in the general theory, we pose several open questions throughout our discussions.

\end{abstract}

\maketitle

\tableofcontents

%%%%%%%%%%%%%%%%%%%%%%%%%%%%%%%%%%%%%%%%%%%%%%%%%%%%%%%%%%%%%%%%%
\section{Introduction}

The methods of least-squares approximation have been used in many areas of engineering since the latter half of the 20th century.  In 1963, Robinson \cite{Ro} studied these methods in the context of digital signal processing and geophysical studies.  His goal was to obtain a finite-length wavelet whose convolution with a given finite-length wavelet best approximates the unit spike.  In other words, given a finite sequence of real numbers $b:=(b_0,b_1,\dots,b_m)$ and a particular value of $n\in\N$, he was trying to find another finite sequence of real numbers $a:=(a_0,a_1,\dots,a_n)$ such that the difference between $a*b$ and the unit spike $(1,0,\dots,0)$ has the smallest $l^2$-norm.  This sequence $a$ was referred to as a least-squares approximate inverse. 

In the 1970s, least-squares approximation appeared in connection with 2D recursive digital filter design.  Applications of 2D digital filtering include the processing of medical pictures, satellite photographs, seismic data mappings, gravity waves data, magnetic recordings, and radar and sonar maps \cite{G}.  In 1972, Justice, Shanks, and Treitel \cite{Sh} applied the methods of least-squares approximation in an effort to ensure 2D filter stability.  Generally speaking, a filter is called stable if it corresponds to a system in which bounded input yields bounded output.  In their studies, they considered the following problem:  given a polynomial $f(z,w)$ of two complex variables, find another polynomial $q(z,w)$ such that the difference between $qf$ and $1$ has the smallest $H^2(\D^2)$-norm.  The polynomial $q$ was referred to as a planar least-squares inverse (PLSI) polynomial.  It is one of the earliest examples of an \textit{optimal polynomial approximant}. 

In 1980, Chui \cite{Ch} formulated the problem of Robinson in the $H^2(\D)$ setting.  More specifically, he considered the following problem:  given $n\in\N$ and a polynomial $f(z):=\sum_{k=0}^mb_kz^k$ with $f(0)\neq 0$, find an $n$th degree polynomial $q(z):=\sum_{k=0}^na_kz^k$ such that the difference between $qf$ and $1$ has the smallest $H^2$-norm.  The polynomial $q$ was referred to as a least-squares inverse polynomial.  For brevity, we'll refer to these as LSI polynomials.  Chui's formulation of the problem ties in well with signal processing since digital filters are often represented by quotients of polynomials.  Now, in many contexts of $1$D recursive digital filter design, it is important to ensure that the poles of the filter are outside of $\overline{\D}$.  If this is the case, then the filter will be stable.  In 1982, Chui and Chan \cite{CC} proposed a design method that guaranteed filter stability.  In the build-up to their method, they reformulated the problem of Robinson to include arbitrary $f\in H^2(\D)$ with $f(0)\neq 0$.  They proceeded to show that for this general class of functions, the LSI polynomials are zero-free in $\overline{\D}$.  

Since the 1980s, least-squares inverse polynomials have been studied in many different function spaces.  In the mathematics community, these polynomials became known as optimal polynomial approximants (OPAs).  In 2015, OPAs were introduced in the Dirichlet-type spaces $\mathcal{D}_{\alpha}$, $\alpha\in\R$ (see \cite{BCLSS} for details).  Since then, several papers have been devoted to investigating the properties of these polynomials (see, e.g., \cite{BKLSS,BKLSS2,BMS}).  Such properties include boundary behavior, connections to reproducing kernel functions, and location of zeros.  In \cite{BKLSS}, it was shown that if $\alpha\geq 0$, $f\in\mathcal{D}_{\alpha}$, and $f(0)\neq 0$, then the zeros of the associated OPA lie outside of $\overline{\D}$; in more recent literature, optimal polynomial approximants and their zeros have been studied in the context of $\ell^p_A$ (see \cite{CRS} for details).

The primary motivation for writing this paper is to investigate the zeros of OPAs in the context of $H^p:=H^p(\D)$, $1\leq p\leq\infty$.  In particular, we would like to know if the nontrivial OPAs are zero-free in the closed disk $\overline{\D}$, which is the case in the Hardy space $H^2$ and the Dirichlet-type spaces $\mathcal{D}_{\alpha}$, $\alpha\geq 0$.  Recall that for $1\leq p<\infty$, $H^p$ is the space of analytic functions $f$ in the open unit disk $\D$ such that 
\begin{equation*}
    \sup_{0\leq r<1}\frac{1}{2\pi}\int_{-\pi}^{\pi}|f(re^{it})|^pdt<\infty.
\end{equation*}
It is well known that $H^p$ is a Banach space with norm
\begin{equation*}
    \|f\|_p:=\bigg\{\sup_{0\leq r<1}\frac{1}{2\pi}\int_{-\pi}^{\pi}|f(re^{it})|^pdt\bigg\}^{\frac{1}{p}}.
\end{equation*}
On the other hand, $H^{\infty}$ is the space of bounded analytic functions in $\D$; this is a Banach space with norm
\begin{equation*}
    \|f\|_{\infty}:=\sup_{z\in\D}|f(z)|.
\end{equation*}

To facilitate our study in the $H^p$ setting, we introduce the notion of optimal polynomial approximant in the space $L^p:=L^p(\T)$; in Section \ref{integralcharacterization}, we will see how OPAs in this context relate to OPAs in the more structured space $L^2$.  For $1\leq p<\infty$, $L^p$ is the space of Lebesgue measurable $\C$-valued functions $f$ on the unit circle $\T$ such that 
\begin{equation*}
\frac{1}{2\pi}\int_{-\pi}^{\pi}|f(e^{it})|^pdt<\infty. 
\end{equation*}
It is well known that $L^p$ is a Banach space with norm
\begin{equation*}
    \|f\|_{p}:=\bigg\{\frac{1}{2\pi}\int_{-\pi}^{\pi}|f(e^{it})|^pdt\bigg\}^{\frac{1}{p}}.
\end{equation*}
On the other hand, $L^{\infty}$ is the space of essentially bounded Lebesgue measurable $\C$-valued functions $f$ on $\T$; this is a Banach space with norm
\begin{equation*}
    \|f\|_{\infty}:=\esssup_{t\in[-\pi,\pi]}|f(e^{it})|.
\end{equation*}

It can be shown that for any $f\in H^p$, $1\leq p\leq\infty$, the nontangential limits $f^*(e^{it})$ exist almost everywhere on $\T$ and $f^*\in L^p$ (see, e.g., \cite{Dp}).  Futhermore, $\|f^*\|_p=\|f\|_p$.  This shows that $H^p$ is isometrically imbedded in $L^p$.  Therefore, we regard any $f\in H^p$ as a function in $L^p$ with the understanding that we are identifying $f$ with its nontangential limit function $f^*$.  Accordingly, any result we obtain for functions in $L^p$ will automatically hold for functions in $H^p$.  

Throughout this paper, we assume that $\N:=\{0,1,2,\dots\}$ and denote $\Pn$ as the space of analytic polynomials of degree at most $n\in\N$; for any $f\in L^p$, we let $f\Pn:=\{fq\colon q\in\Pn\}$.  We begin the paper by discussing existence and uniqueness of optimal polynomial approximants in $L^p$, $1\leq p\leq\infty$; for the extreme cases of $p=1$ and $p=\infty$, there seems to be a lot of fertile ground for research.  We then introduce OPAs in the Hilbert space $L^2$; this is a generalization of the familiar theory in $H^2$ (see \cite{BC} for a survey of OPAs in $\mathcal{D}_{\alpha}$).  In Section \ref{integralcharacterization}, we use an orthogonality condition to characterize OPAs in $L^p$, $1< p<\infty$; we later use this condition to compute some OPAs for several values of $p$.  Throughout this paper, we discuss what we know and pose several open questions about the zeros of optimal polynomial approximants in $L^p$.

%%%%%%%%%%%%%%%%%%%%%%%%%%%%%%%%%%%%%%%%%%%%%%%%%%%%%%%%%%%%%%%%%%%%
\section{Existence and Uniqueness}

The context of this paper will be the space $L^p$, $1\leq p\leq\infty$; we would like to have an adequate structure to support the theory, so it seems natural to consider only the values of $p$ for which $L^p$ is a Banach space.  In the context of $L^p$, existence and uniqueness of best approximations can be a delicate issue.  However, if $Y$ is a finite dimensional subspace of $L^p$, then there exists a best approximation to any $g\in L^p$ out of $Y$ (see, e.g., \cite[Theorem 6.1-1]{K}); this allows us to define the notion of optimal polynomial approximant.

%The notion of optimal polynomial approximant has been studied extensively in the context of the Dirichlet-type spaces $\mathcal{D}_{\alpha}$, $\alpha\in\R$.  

%For any $f\in\mathcal{D}_{\alpha}\setminus\{0\}$ and $n\in\N$, the $n$th OPA of $1/f$ in $\mathcal{D}_{\alpha}$ is defined as the unique polynomial that minimizes the norm $\|qf-1\|_{\alpha}$, where $q$ varies over $\mathcal{P}_n$.  Here, $\mathcal{P}_n$ denotes the space of polynomials of degree at most $n$.  In this case, it is easy to see that such a polynomial must exist and that it is unique:  since $\mathcal{D}_{\alpha}$ is a Hilbert space, the orthogonal projection of $1$ onto the subspace $f\mathcal{P}_n:=\{fq\colon q\in\Pn\}$ gives the unique function $fQ$ that minimizes the distance from $1$.  Since $f$ is not identically zero, the polynomial $Q$ must also be unique.

%In less structured normed spaces, existence and uniqueness of best approximations become more of a delicate issue.  

            \begin{prop}[Existence]\label{Existence}
Let $1\leq p\leq\infty$ and $n\in\N$.  For $f,g\in L^p$, $f\not\equiv 0$, there exists a polynomial $Q\in\Pn$ such that 
\[
\inf_{q\in\mathcal{P}_n}\|qf-g\|_p=\|Q f-g\|_p.
\]
\end{prop}
\begin{proof}
Since $f\Pn$ is a finite dimensional subspace of $L^p$, the result is immediate.
\end{proof}

Proposition \ref{Existence} states that there is a best approximation to $g$ out of the space $f\mathcal{P}_n$.  If $fQ$ is such a best approximation, then the polynomial $Q\in\Pn$ is referred to as an \textit{optimal polynomial approximant in $L^p$}.  Now, it's important to note that best approximations in the Banach space setting are not always unique.  However, $L^p$ is a strictly convex for $1<p<\infty$.  Therefore, for any given subspace $Y$ of $L^p$, there is at most one best approximation to any $g\in L^p$ out of $Y$ (see, e.g., \cite[Theorem 6.2-3]{K}).  This guarantees that optimal polynomial approximants in $L^p$, $1<p<\infty$, are unique. 
(See \cite{ST} for a similar discussion of uniqueness of OPAs in $\ell^p_A$.)

%(For information about uniqueness of OPAs in $\ell^p_A$, see \cite{ST}.)

%For example, consider the space of points in $\R^2$ with norm $\|(x,y)\|:=\max\{|x|,|y|\}$.  If $Y$ denotes the subspace consisting of all points on the line $x=1$, then any point on the line segment from $(1,-1)$ to $(1,1)$ is a best approximation to $(0,0)$ out of $Y$.

%On a different note, suppose that $(X, \|\cdot\|)$ is a strictly convex normed space.  

            \begin{prop}[Uniqueness]\label{Uniqueness}
Let $1<p<\infty$ and $n\in\N$.  For $f,g\in L^p$, $f\not\equiv 0$, there exists a unique polynomial $Q\in\Pn$ such that 
\begin{equation*}
    \inf_{q\in\mathcal{P}_n}\|qf-g\|_p=\|Q f-g\|_p. 
\end{equation*}
\end{prop}
\begin{proof}
It is well-known that $L^p$ is strictly convex for $1<p<\infty$.  Since $g$ has a best approximation out of $f\Pn$, it must be unique.
\end{proof}

In the case of $L^1$, Proposition \ref{Existence} implies that the set of best approximations to a given function $g$ out of $f\mathcal{P}_n$ is nonempty.  Now, since the set of best approximations must be convex (see, e.g., \cite[Lemma 6.2-1]{K}), there is either a unique best approximation or an infinite number of best approximations.  The following proposition shows that the former does not need to hold.

\begin{prop}\label{uniquenesseg1}
Optimal polynomial approximants in $L^1$ are not necessarily unique.
\end{prop}

\begin{proof}
Let $f\equiv 1$, and consider the characteristic function $g(e^{it}):=\chi_{[-\pi,0)}(t)$.  We show that 
\begin{equation*}
    \inf_{q\in\mathcal{P}_0}\|qf-g\|_1=\|af-g\|_1
\end{equation*}
for all $a\in[0,1]$.  Let $\Lambda:=\{h\in L^{\infty}\colon \|h\|_{\infty}\leq 1\;\text{and}\;\frac{1}{2\pi}\int_{-\pi}^{\pi}h dt=0\}$.
By Hahn-Banach duality, it follows that 
\begin{equation*}
    \inf_{q\in\mathcal{P}_0}\|qf-g\|_1=\sup_{h\in\Lambda}\bigg|\frac{1}{2\pi}\int_{-\pi}^{\pi}gh dt\bigg|.
\end{equation*}
If we let $h_0(e^{it}):=\chi_{[-\pi,0)}(t)-\chi_{[0,\pi)}(t)$, then we see that $h_0\in\Lambda$.  Consequently, 
\begin{equation}\label{ineq1}
    \inf_{q\in\mathcal{P}_0}\|qf-g\|_1\geq \bigg|\frac{1}{2\pi}\int_{-\pi}^{\pi}gh_0 dt\bigg|=\frac{1}{2}.
\end{equation}
Now, suppose that $a\in[0,1]$.  Then
\begin{align}\label{eqq2}
    \|af-g\|_1&=\frac{1}{2\pi}\int_{-\pi}^{\pi}|a-g|dt\nonumber\\
    &=\frac{1}{2\pi}\int_{-\pi}^{0}(1-a)dt+\frac{1}{2\pi}\int_{0}^{\pi}a dt\nonumber\\
    &=\frac{1}{2}.
\end{align}
From (\ref{ineq1}) and (\ref{eqq2}), we conclude that
\begin{equation*}
    \inf_{q\in\mathcal{P}_0}\|qf-g\|_1=\|af-g\|_1
\end{equation*}
whenever $a\in[0,1]$.
\end{proof}

On the other hand, consider the case of $L^{\infty}$.  From Proposition \ref{Existence} and our previous discussion, the set of best approximations to a given function $g$ out of $f\Pn$ consists of either a single element or an infinite number of elements.  As in the case of $L^1$, the former does not need to hold.

                \begin{prop}\label{uniquenesseg}
Optimal polynomial approximants in $L^{\infty}$ are not necessarily unique.
\end{prop}
\begin{proof}
Let $f(z)=1-z$.  We show that 
\begin{equation*}
\inf_{q\in\mathcal{P}_0}\|qf-1\|_{\infty}=\|af-1\|_{\infty}
\end{equation*}
for all $a\in[0,1]$.  For any $a\in\C$,
\begin{equation}
    \|af-1\|_{\infty}=\|(a-1)-az\|_{\infty}\leq  |a-1|+|a|. \label{equ}
\end{equation}
If $a=0$, then equality in (\ref{equ}) clearly holds.  If $a\neq 0$, let
\begin{equation*}
    w:=-\frac{|a|}{a}e^{i\arg(a-1)}.
\end{equation*}
Note that 
\begin{align*}
    |af(w)-1|&=|(a-1)-aw|\\
    &=\big||a-1|e^{i\arg(a-1)}+|a|e^{i\arg(a-1)}\big|\\
    &=|a-1|+|a|.
    \end{align*}
Thus, equality in (\ref{equ}) holds, and it follows that 
\begin{equation*}
    \|af-1\|_{\infty}=|a+1|+|a|\geq (1-|a|)+|a|=1
\end{equation*}
for all $a\in\C$.  Consequently,
\begin{equation}\label{inf1}
\inf_{q\in\mathcal{P}_0}\|qf-1\|_{\infty}=1,    
\end{equation}
where the infimum is attained for all $a\in\C$ such that $|a-1|+|a|=1$.

If $a\in\C$ satisfies $|a-1|+|a|=1$, then
\begin{equation*}
    |1-a|+|a|=|(1-a)+a|.
\end{equation*}
Hence, $a=\alpha(1-a)$ for some positive real $\alpha$.  It follows that $a$ is a real number in $[0,1]$.  Conversely, if $a\in[0,1]$, then 
\begin{equation*}
    |a-1|+|a|=(1-a)+a=1.
\end{equation*}
Therefore, we conclude that $|a-1|+|a|=1$ if and only if $a\in [0,1]$.  It follows that 
\begin{equation}\label{inf2}
    \|af-1\|_{\infty}=1\quad\text{for all}\quad a\in[0,1].
\end{equation}
By combining (\ref{inf1}) and (\ref{inf2}), the result follows.
\end{proof}

Although optimal polynomial approximants in $L^p$ are not necessarily unique for $p=1$ or $p=\infty$, there are still interesting questions that one can ask about best approximations in these spaces.  For example, when there are infinitely many best approximations to $g$ out of $f\Pn$, is there an optimal polynomial approximant that has norm strictly smaller than all the others?  In the case of Proposition \ref{uniquenesseg1} (or Proposition \ref{uniquenesseg}), this optimal polynomial approximant is clearly the constant $0$.  As another example, is there a way to characterize the functions $f$ such that the constant function $0$ is a best approximation to $1$ out of $f\mathcal{P}_n$?  In Proposition \ref{propzero}, we will see an analogous characterization in the context of $H^p$, $1<p<\infty$.  Nevertheless, a best approximation to $g$ out of $f\Pn$ is unique whenever $1<p<\infty$.  This leads us to the following definition.  

            \begin{defn}[$n$th OPA]
Let $1<p<\infty$ and $n\in\N$.  For $f,g\in L^p$, $f\not\equiv 0$, the unique polynomial that minimizes the norm $\|qf-g\|_p$, where $q$ varies over $\Pn$, is denoted by $q_{n,p}[f,g]$ and is called the $n$th OPA of $g/f$ in $L^p$.

%Let $1<p<\infty$ and $n\in\N$.  For $f,g\in L^p$, $f\not\equiv 0$, the unique polynomial in Proposition \ref{Uniqueness} is denoted by $q_{n,p}[f,g]$ and is called the $n$th OPA of $g/f$ in $L^p$.
\end{defn}

In the next section, we develop the theory of OPAs in $L^2$; the inner product structure gives us an efficient way to compute the coefficients of OPAs; it also gives us an easy way to verify whether or not the OPA $q_{1,2}[f,g]$ is zero-free in $\overline{\D}$.  As we will see in Section \ref{integralcharacterization}, it is possible to relate certain OPAs in $L^p$, $p\geq2$, to OPAs in $L^2$.

%%%%%%%%%%%%%%%%%%%%%%%%%%%%%%%%%%%%%%%%%%%%%%%%%%%%%%%%%%%%%%%%%%%%%
\section{OPAs in \texorpdfstring{$L^2$}{TEXT}}\label{sec3}

In 1980, Chui \cite{Ch} presented the idea of least-squares inverse polynomials in the context of $H^2$.  The LSI polynomials being studied can be viewed as OPAs of the form $q_{n,2}[f,1]$, where $f$ is a polynomial.  In 1982, Chui and Chan \cite{CC} presented a filter design method that extended this study to include arbitrary functions $f\in H^2$, $f(0)\neq 0$.  The reason for considering more general $f$ was to approximate certain outer functions that were defined in conjunction with an ideal digital filter.  One of the selling points of this method was that the coefficients of the LSI polynomial could be expressed as the solution to a system of linear equations.  Moreover, the associated matrix is positive definite and Toeplitz.  (In fact, it's a Gram matrix.) Therefore, its inverse can be computed quickly with any of the available algorithms (see \cite{A} for details). 

In this section, we'll see that some of the most important OPA properties in $H^2$ (particularly the fact that the coefficients satisfy a system of linear equations) generalize naturally to the $L^2$ setting.  Before we begin, let's note that the following proposition follows immediately from the notion of orthogonal projection.

    \begin{prop}[Linearity]\label{props}
Let $f,g,h\in L^2$, $f\not\equiv 0$, $n\in\N$, and $\alpha,\beta\in\C$.  Then
\begin{equation*}
   q_{n,2}[f,\alpha g+\beta h]=\alpha q_{n,2}[f,g]+\beta q_{n,2}[f,h].
   \end{equation*}

\begin{proof}
By definition, the OPA $q_{n,2}[f,g]$ minimizes the norm $\|qf-g\|_2$, where $q$ varies over $\mathcal{P}_n$.  Since $L^2$ is a Hilbert space and $f\mathcal{P}_n$ is a closed subspace, it follows that $q_{n,2}[f,g]f$ is the orthogonal projection of $g$ onto $f\mathcal{P}_n$.  The proposition follows by linearity of the projection.
\end{proof}

\end{prop}

In the following result, we characterize OPAs of the form $q_{n,2}[f,g]$, where $f,g\in L^2$.  

            \begin{prop}\label{ortho}
Let $f,g\in L^2$, $f\not\equiv 0$, $n\in \N$, and $Q\in\Pn$.  Then
\begin{equation*}
    \langle Qf-g,z^kf\rangle=0
\end{equation*}
for $k=0, \dots, n$ if and only if $Q=q_{n,2}[f,g]$.
\end{prop}
\begin{proof}
Let $q_{n,2}:=q_{n,2}[f,g]$.  If $\langle Qf-g,z^kf\rangle=0$ for $k=0,\dots,n$, then $(Qf-g)\perp f\mathcal{P}_n$.  Since $g=Qf+(g-Qf)$, it follows that $Qf$ is the projection of $g$ onto $f\Pn$, i.e., $Q=q_{n,2}$.  Conversely, since $q_{n,2}f$ is the projection of $g$ onto $f\Pn$, and since $g=q_{n,2}f+(g-q_{n,2}f)$, it follows that $(q_{n,2}f-g)\in L^2\ominus f\Pn$.  In particular, $\langle q_{n,2}f-g,z^kf\rangle=0$ for $k=0,\dots,n$.
\end{proof}

A valuable aspect of Proposition \ref{ortho} is that it allows us to compute the coefficients of $q_{n,2}[f,g]$ in an efficient way; as we will see, the coefficients are expressed as a solution to a system of linear equations.  In fact, the associated matrix has the same properties as the one that Chui and Chan were considering in \cite{CC}.  (For a recent discussion of the computation of coefficients in the context of analytic function spaces, see \cite{BIMS,BMS}.) 

%it's interesting to note that this matrix is independent of the choice of $g$. 

            \begin{theorem}[Coefficients in $L^2$]\label{coeff}
Let $f,g\in L^2$, $f\not\equiv 0$, and $n\in\N$.  Set $q_{n,2}[f,g](z)=\sum_{j=0}^na_jz^j$.  Then the coefficients of $q_{n,2}[f,g]$ satisfy
\begin{equation*}
    \delimitershortfall=3pt
    \begin{bmatrix}
    \langle f,f\rangle & \langle zf,f\rangle &\dots&\langle z^nf,f\rangle\\
    \langle f,zf\rangle&\langle zf,zf\rangle&\dots&\langle z^nf,zf\rangle\\
    \vdots&\vdots&&\vdots\\
    \langle f,z^nf\rangle&\langle zf,z^nf\rangle&\dots&\langle z^nf,z^nf\rangle
    \end{bmatrix}
    \begin{bmatrix}
    a_0\\
    a_1\\
    \vdots\\
    a_n
    \end{bmatrix}
    =
    \begin{bmatrix}
    \langle g,f\rangle\\
    \langle g, zf\rangle\\
    \vdots\\
    \langle g,z^nf\rangle
    \end{bmatrix}.
\end{equation*}
\end{theorem}
\begin{proof}
From Proposition \ref{ortho}, it follows that
\begin{equation}
    \sum_{j=0}^n\langle z^jf,z^kf\rangle a_j=\langle g,z^kf\rangle \quad\text{for}\quad k=0,\dots,n. \label{eq1}
\end{equation}
Let $a$ denote the vector with components $a_j$, $j=0,\dots, n$,  and $y$ denote the vector with components $\langle g,z^kf\rangle$, $k=0,\dots,n$.  If $B$ is the $(n+1)\times(n+1)$ matrix with entries $b_{kj}=\langle z^jf,z^kf\rangle$, $0\leq k,j\leq n$, it is easy to see that (\ref{eq1}) is equivalent to $Ba=y$.
\end{proof}

%It's worth noting that this result can be formulated in the context of $L^2(\mu)$, where $d\mu:=\frac{1}{2\pi}|f|^2dt$.  

Let's look at an example to demonstrate the computation of coefficients in $L^2$.

                    \begin{ex}\label{ex3.4}
Let $f(z)=1-z$.  Then $q_{1,2}[f,1](z)=\frac{1}{3}z+\frac{2}{3}$.
\end{ex}

\begin{proof}
Set $q_{1,2}[f,1](z)=a_1z+a_0$.  By Theorem \ref{coeff}, the coefficients of $q_{1,2}[f,1](z)$ satisfy
\begin{equation*}
    \begin{bmatrix}
    2&-1\\
    -1&2
    \end{bmatrix}
    \begin{bmatrix}
    a_0\\
    a_1
    \end{bmatrix}
    =
    \begin{bmatrix}
    1\\
    0
    \end{bmatrix}.
\end{equation*}
Therefore, it easily follows that $q_{1,2}[f,1](z)=\frac{1}{3}z+\frac{2}{3}$.
\end{proof}

As previously suggested, the design method of Chui and Chan was motivated by the fact that all nontrivial LSI polynomials are zero-free in $\overline{\D}$.  Since we are extending these polynomials to the $L^2$ setting, it's natural to question if OPAs of the form $q_{n,2}[f,g]$, where $f,g\in L^2$, are zero-free in $\overline{\D}$.  It turns out that this is not the case; the following theorem shows that the function $g$ plays a prominent role.

%with this in mind, one could design a digital filter whose poles are outside of $\overline{\D}$ (a stable filter). 

%To put it differently, if $f\in H^2$ and $f(0)\neq 0$, then $q_{n,2}[f,1]$ does not vanish in $\overline{\D}$.  

%Obtaining such a filter is an important aspect of the design process. Filters that are not stable could produce output that increases without bound, which could be problematic in practice.

%In the general context of $L^2$, it's natural to question if OPAs of the form $q_{n,2}[f,g]$ are zero-free in $\overline{\D}$.  

            \begin{theorem}\label{condition}
Let $f,g\in L^2$ and $f\not\equiv 0$.  The OPA $q_{1,2}[f,g]$ is zero-free in $\overline{\D}$ if and only if 
\begin{equation*}
    |\langle g,f\rangle|>|\langle g,zf\rangle|.
\end{equation*}
\end{theorem}
\begin{proof}
Set $q_{1,2}[f,g](z)=a_1z+a_0$.  From Proposition \ref{ortho}, we have that 
\begin{equation}
    a_0\langle f,z^kf\rangle+a_1\langle zf,z^kf\rangle=\langle g,z^kf\rangle\quad\text{for}\quad k=0,1.\label{eq3}
\end{equation}
This is equivalent to
\begin{equation}
    \begin{bmatrix}
    \langle f,f\rangle&\langle zf,f\rangle\\
    \langle f,zf\rangle&\langle zf,zf\rangle
    \end{bmatrix}
    \begin{bmatrix}
    a_0\\
    a_1
    \end{bmatrix}
    =
    \begin{bmatrix}
    \langle g,f\rangle\\
    \langle g,zf\rangle
    \end{bmatrix}.\label{eq2}
\end{equation}
Let $A$ denote the $2\times 2$ matrix in (\ref{eq2}).  Since $|\langle f,f\rangle|>|\langle f,zf\rangle|$, and since the shift operator is an isometry on $H^2$, it follows that $\det A\neq 0$.  Therefore,
\begin{equation*}
    a_0=\frac{\langle zf,zf\rangle \langle g,f\rangle-\langle zf,f\rangle\langle g,zf\rangle}{\det A}
\end{equation*}
and
\begin{equation*}
    a_1=\frac{\langle f,f\rangle\langle g,zf\rangle-\langle f,zf\rangle\langle g,f\rangle }{\det A}.
\end{equation*}
\item[Case 1 ($a_1=0$):]
If $q_{1,2}[f,g]$ is zero-free in $\overline{\D}$, then $a_0\neq 0$.  It follows from (\ref{eq3}) that
\begin{align*}
    |\langle g,f\rangle|=|\langle a_0f,f\rangle|
    >|\langle a_0f,zf\rangle|
    =|\langle g,zf\rangle|.
\end{align*}
Conversely, suppose that $|\langle g,f\rangle|>|\langle g,zf\rangle|$.  Then $\langle g,f\rangle$ and $\langle g,zf\rangle$ cannot simultaneously be zero. From $(\ref{eq2})$, it follows that $a_0\neq 0$, i.e., $q_{1,2}[f,g]$ is zero-free in $\overline{\D}$.
\item[Case 2 ($a_1\neq 0$):]
Let $q_{1,2}[f,g](z_0)=0$.  Then
\begin{align}
    z_0=-\frac{a_0}{a_1}=\frac{\langle zf,f\rangle\langle g,zf\rangle-\langle zf,zf\rangle\langle g,f\rangle}{\langle f,f\rangle\langle g,zf\rangle-\langle f,zf\rangle\langle g,f\rangle}.\label{eq4}
\end{align}
If $\langle g,zf\rangle=0$, then it follows from (\ref{eq2}) that $\langle g,f\rangle\neq 0$.  From (\ref{eq4}), we see that
\begin{equation*}
    |z_0|=\frac{|\langle zf,zf\rangle|}{|\langle f,zf\rangle|}>1.
\end{equation*}
Therefore, both implications in the theorem are trivially true.  We thus assume that $\langle g,zf\rangle\neq 0$.  Let $w:=\frac{\langle g,f\rangle}{\langle g,zf\rangle}$ and $\alpha:=\frac{\langle zf,f\rangle}{\langle zf,zf\rangle}$.  From (\ref{eq4}), it follows that 
\begin{equation*}
    |z_0|=\bigg|\frac{w-\alpha}{1-\overline{\alpha}w}\bigg|.
\end{equation*}
Since $\alpha\in\D$, $z_0\in\overline{\D}$ if and only if $w\in\overline{\D}$.  Equivalently, $q_{1,2}[f,g]$ is zero-free in $\overline{\D}$ if and only if $|\langle g,f\rangle|>|\langle g,zf\rangle|$.
\end{proof}

The following example demonstrates the relative ease of applying Theorem \ref{condition}.

                \begin{ex}[Zero in $\overline{\D}$]\label{ex6}
Let $f(z)=z^2+1$ and $g(z)=2z-1$.  Note that 
\begin{equation*}
    |\langle g,f\rangle|=1<2=|\langle g,zf\rangle|.
\end{equation*}
By Theorem \ref{condition}, $q_{1,2}[f,g]$ has its zero in $\overline{\D}$.  

%More specifically, by setting $q_{1,2}[f,g](z)=a_1z+a_0$, it follows from Theorem \ref{coeff} that
%\begin{equation*}
 %   \begin{bmatrix}
  %  2& 0\\
   % 0& 2
    %\end{bmatrix}
%    \begin{bmatrix}
 %   a_0\\
  %  a_1
   % \end{bmatrix}
%    =
 %   \begin{bmatrix}
  %  -1\\
   % 2
    %\end{bmatrix}.
%\end{equation*}
%Therefore, $q_{1,2}[f,g](z)=z-\frac{1}{2}$.
\end{ex}

          %      \begin{ex}[Zero on $\T$]
%Let $f(z)=1-z$ and $g(z)=z+2$.  Note that
%\begin{equation*}
 %   |\langle g,f\rangle|=1=|\langle g,zf\rangle|.
%\end{equation*}
%By Theorem \ref{condition}, $q_{1,2}[f,g]$ must have its zero in $\overline{\D}$.  An easy calculation shows that $q_{1,2}[f,g](z)=z+1$.
%\end{ex}

            %    \begin{ex}[Zero-free in $\overline{\D}$]\label{ex8}
%Let $f(z)=2z+1$ and $g(z)=z^2-z-1$.  Since
%\begin{equation*}
 %   |\langle g,f\rangle|=3>1=|\langle g,zf\rangle|,
%\end{equation*}
%it follows from Theorem \ref{condition} that $q_{1,2}[f,g]$ is zero-free in $\overline{\D}$.  In fact, $q_{1,2}[f,g](z)=\frac{11}{21}z-\frac{17}{21}$.
%\end{ex}

%Given any functions $f,g\in L^2$, $f\not\equiv 0$, Theorem \ref{condition} makes it easy to determine whether or not the OPA $q_{1,2}[f,g]$ is zero-free in $\overline{\D}$; as we will see in the upcoming sections, this will give us insight into the zeros of certain OPAs in $L^p$. 

It's worth noting that the results in this section can be formulated in the context of $L^2(\mu)$, where $d\mu:=\frac{1}{2\pi}|f|^2dt$.  At any rate, in the case where $n$ is arbitrary, the following is still unknown.

\begin{OQ}\label{oq1}
Let $n\in\N$.  For which functions $f,g\in L^2$ can we guarantee that $q_{n,2}[f,g]$ is zero-free in $\overline{\D}$?
\end{OQ}

%%%%%%%%%%%%%%%%%%%%%%%%%%%%%%%%%%%%%%%%%%%%%%%%%%%%%%%%%%%%
\section{Orthogonality Condition}\label{integralcharacterization}

Our main motivation for writing this paper is to investigate the zeros of optimal polynomial approximants in the context of $H^p$.  For the Hilbert space $H^2$, we already know that OPAs of the form $q_{n,2}[f,1]$, where $f\in H^2$ and $f(0)\neq 0$, are zero-free in $\overline{\D}$; one proof of this follows from a special case of Proposition \ref{ortho}.  Now, in an effort to understand the zeros of OPAs in $H^p$, we generalize Proposition \ref{ortho} to functions in $L^p$.  As a result, this allows us to relate OPAs in $L^p$ to OPAs in the more structured space $L^2$.

Recall that in a normed space $(X,\|\cdot\|)$, an element $x\in X$ is said to be \textit{orthogonal} to a subspace $Y$ if $\|x\|\leq \|x+y\|$ for all $y\in Y$.  In the case where $Y$ is a subspace of $L^p$, $1<p<\infty$, Shapiro \cite[Theorem 4.2.1]{Hsh} characterized the functions $f\in L^p$ that are orthogonal to $Y$.  Now, it follows from the definition of OPA that $Qf-g$ is orthogonal to $f\Pn$ if and only if $Q=q_{n,p}[f,g]$.  This observation, along with Shapiro's result, allows us to generalize Proposition \ref{ortho}.  We refer to this generalization as the \textit{orthogonality condition in $L^p$}.

%For the Hilbert space $H^2$, we already know that OPAs of the form $q_{n,2}[f,1]$, where $f\in H^2$ and $f(0)\neq 0$, are zero-free in $\overline{\D}$.  The proof of this follows from a special case of Proposition \ref{ortho}.  Now, in an effort to understand the zeros of OPAs in $H^p$, we generalize Proposition \ref{ortho} to functions in $L^p$.  As a result, this allows us to relate OPAs in $L^p$ to OPAs in $L^2$.  Consequently, we can study the properties of OPAs in $L^p$, particularly the location of their zeros, by using the additional tools that the Hilbert space structure provides.  Needless to say, results in $H^p$ will follow as a special case of those in $L^p$.

%Recall that in a normed space $(X,\|\cdot\|)$, an element $x\in X$ is said to be orthogonal to a subspace $Y$ if $\|x\|\leq \|x+y\|$ for all $y\in Y$.  In the case of $L^p$, it follows from the definition of OPA that $Qf-g$ is orthogonal to $f\Pn$ if and only if $Q=q_{n,p}[f,g]$.  This observation allows us to generalize Proposition \ref{ortho}.  We refer to this generalization as the \textit{orthogonality condition in $L^p$}.

                \begin{prop}[Orthogonality Condition]\label{orthop}
Let $1<p<\infty$, $f,g\in L^p$, $f\not\equiv 0$, $n\in\N$, and $Q\in\Pn$.  Then
\begin{equation*}
    \frac{1}{2\pi}\int_{-\pi}^{\pi}|Qf-g|^{p-1}\sgn(Qf-g)e^{-ikt}\overline{f} dt=0
\end{equation*}
for $k=0,\dots,n$ if and only if $Q=q_{n,p}[f,g]$.
\end{prop}

\begin{proof}
Let $F:=Qf-g$. From \cite[Theorem 4.2.1]{Hsh},
\begin{equation}\label{sgneq}
    \frac{1}{2\pi}\int_{-\pi}^{\pi}|F|^{p-1}\overline{\sgn(F)}qf dt=0
\end{equation}
for all $q\in\Pn$ if and only if $F$ is orthogonal to $f\Pn$.  From our previous discussion, it follows that (\ref{sgneq}) is true for all $q\in\Pn$ if and only if $Q=q_{n,p}[f,g]$.  This is equivalent to our assertion.
\end{proof}

%One of the primary reasons Proposition \ref{ortho} is useful is that it allows us to easily compute the coefficients of OPAs in $L^2$.  In fact, we can express the coefficients as the solution to a system of linear equations for which the associated matrix can easily be inverted.  Since Proposition \ref{orthop} is a generalization of Proposition \ref{ortho}, it's natural to wonder if Proposition \ref{orthop} can be used to readily compute the coefficients of OPAs in $L^p$.  By imposing certain conditions on $f$ and $p$, we will see in Section \ref{comp} that this is the case. 

%On a different note, Proposition \ref{orthop} can give us insight into OPAs of the form $q_{n,p}[f,1]$, where $2\leq p<\infty$.  Before we look into the details, consider the following remark.

One of the primary reasons Proposition \ref{ortho} is useful is that it allows us to easily compute the coefficients of OPAs in $L^2$; since Proposition \ref{orthop} is a generalization of Proposition \ref{ortho}, it's natural to wonder if Proposition \ref{orthop} can be used to readily compute the coefficients of OPAs in $L^p$.  By imposing certain conditions on $f$ and $p$, we will see in Section \ref{comp} that this is the case.  

On a different note, Proposition \ref{orthop} can give us insight into OPAs of the form $q_{n,p}[f,1]$, where $2\leq p<\infty$.  Before we look into the details, consider the following remark.

\begin{rem}\label{remark}
Let $2\leq p<\infty$, $f\in L^p\setminus\{0\}$, and $n\in\N$.  Define the function $g:=\{q_{n,p}[f,1]f-1\}^{\frac{p-2}{2}}$.  If $g\equiv 0$, then $q_{n,p}[f,1]$ is zero-free on $\T$.  Moreover, if $f\in H^p$ and $g\equiv 0$, then $q_{n,p}[f,1]$ is zero-free in $\overline{\D}$.
\end{rem}

We will reference this remark when we study the location of zeros in the next section.  Nevertheless, perhaps the most valuable insight into the zeros of OPAs can be obtained by the following theorem.

                \begin{thm}\label{cor4}
Let $2\leq p<\infty$, $f\in L^p\setminus\{0\}$, $n\in\N$, and $g:=\{q_{n,p}[f,1]f-1\}^{\frac{p-2}{2}}$.  If $g\not\equiv 0$, then
\begin{equation*}
    q_{n,p}[f,1]=q_{n,2}[fg,g].
\end{equation*}
\end{thm}

\begin{proof}
Let $Q:=q_{n,p}[f,1]$.  By Proposition \ref{orthop}, we have that 
\begin{equation*}
    \frac{1}{2\pi}\int_{-\pi}^{\pi}|Qf-1|^{p-1}\sgn (Qf-1)e^{-ikt}\overline{f}dt=0
\end{equation*}
for $k=0,\dots,n$.  Now, we can write
\begin{align*}
   \frac{1}{2\pi}\int_{-\pi}^{\pi}|Qf-1|&^{p-1}\sgn (Qf-1)e^{-ikt}\overline{f}dt\\
   &=\frac{1}{2\pi}\int_{-\pi}^{\pi}|Qf-1|^{p-2}(Qf-1)e^{-ikt}\overline{f}dt\\
   &=\frac{1}{2\pi}\int_{-\pi}^{\pi}(Qf-1)^{\frac{p-2}{2}}(\overline{Qf}-1)^{\frac{p-2}{2}}(Qf-1)e^{-ikt}\overline{f}dt\\
   &=\frac{1}{2\pi}\int_{-\pi}^{\pi}g(Qf-1)e^{-ikt}\overline{fg}dt\\
   &=\langle Qfg-g,z^kfg\rangle.
\end{align*}
Thus, $\langle Qfg-g,z^kfg\rangle=0$ for $k=0,\dots,n$.  By Proposition \ref{ortho}, we conclude that $Q=q_{n,2}[fg,g]$.
\end{proof}

%The beauty of this theorem lies in its ability to relate OPAs in the Banach space $L^p$ to OPAs in the Hilbert space $L^2$.  Consequently, 

Theorem \ref{cor4} suggests that we can gain a better understanding of OPAs in $L^p$ by studying OPAs of the form $q_{n,2}[fg,g]$, where $f$ and $g$ are arbitrary functions in $L^2$ such that $fg\in L^2\setminus\{0\}$; this amounts to studying the unique polynomial that minimizes the norm $\|\{qf-1\}g\|_2$, where $q$ varies over $\mathcal{P}_n$.  As a particular point of interest, we would like to know the necessary conditions to impose on $g$ in order for $q_{n,2}[fg,g]$ to be zero-free in $\overline{\D}$.  The following example shows that $g$ cannot be arbitrary.

               \begin{ex}[Zero in $\D$]\label{4.5}
Let $f(z)=2z+1$ and $g(z)=1-z$.  Then $q_{1,2}[fg,g]=-\frac{6}{35}z-\frac{1}{35}$.
\end{ex}

                 %\begin{ex}[Zero on $\T$]\label{4.6}
%Let $f(z)=1-z$ and $g(z)=z+1$.  Then $q_{1,2}[fg,g](z)=\frac{1}{2}z+\frac{1}{2}$.
%\end{ex}

            %\begin{ex}[Zero-free in %$\overline{\D}$]\label{ex4.7}
%Let $f(z)=z-2$ and $g(z)=z-3$.  Then $q_{1,2}[fg,g](z)=-\frac{433}{2619}z-\frac{1216}{2619}$.
%\end{ex}

To better align with our main motivation, let's consider the case where $f\in H^p$.  If we assume that $2\leq p<\infty$, then it's easy to see that the function $g$ in Theorem \ref{cor4} belongs to $L^2$.  However, it's not at all clear if $g$ belongs to $H^2$.  Nevertheless, the following proposition shows that the OPA $q_{n,2}[fg,g]$ is equal to an OPA obtained by replacing $g$ with a particular function in $H^2$.

                             \begin{thm}\label{4.7}
Let $2\leq p<\infty$, $f\in H^p\setminus\{0\}$, $n\in\N$, and $g:=\{q_{n,p}[f,1]f-1\}^{\frac{p-2}{2}}$.  If $g\not\equiv 0$, then there exists a function $h\in H^2$, which is zero-free in $\D$, such that 
\begin{equation*}
    q_{n,p}[f,1]=q_{n,2}[fh,h].
\end{equation*}
\end{thm}

\begin{proof}
Define the function $g:=\{q_{n,p}[f,1]f-1\}^{\frac{p-2}{2}}$.  Note that we can write $q_{n,p}[f,1]f-1=u\Phi$, where $u$ is an inner function and $\Phi$ is an outer function.  Let $h:=\Phi^{\frac{p-2}{2}}$.  Since $\Phi$ is zero-free in $\D$, it follows that $h\in H^2$.  Furthermore, $h$ is zero-free in $\D$.  Now, since $g=u^\frac{p-2}{2}h$, we see that
\begin{align*}
    \|q_{n,2}[fh,h]fg-g\|_2&=\|q_{n,2}[fh,h]fh-h\|_2\\
    &=\inf_{q\in\Pn}\|qfh-h\|_2\\
    &=\inf_{q\in\Pn}\|qfg-g\|_2\\
    &=\|q_{n,2}[fg,g]fg-g\|_2.
\end{align*}
By uniqueness of OPAs and Theorem \ref{cor4}, the result follows.

\end{proof}

Theorem \ref{4.7} indicates that it's worth studying OPAs of the form $q_{n,2}[fg,g]$, where $f$ and $g$ are arbitrary functions in $H^2$ such that $fg\in H^2\setminus\{0\}$.  Accordingly, we would like to address a modified version of Open Question \ref{oq1} .

\begin{OQ}
Let $n\in\N$.  For which functions $f,g\in H^2$ can we guarantee that $q_{n,2}[fg,g]$ is zero-free in $\overline{\D}$?
\end{OQ}

\section{Location of Zeros}\label{locationofzeros}
Our motivation to study the zeros of optimal polynomial approximants comes from the early application of PLSI polynomials to digital filter design.  One of the most important aspects to designing a recursive digital filter is ensuring filter stability; recall that a filter is stable if it corresponds to a system in which bounded input yields bounded output.  In 2D signal processing, a stable filter can be obtained by a rational function of two complex variables such that the denominator is zero-free in $\overline{\D^2}$.  In 1972, Justice, Shanks, and Treitel \cite{Sh} proposed that such a rational function can be created with the use of planar least-squares inverse polynomials.  Accordingly, they made the following conjecture:  for any polynomial $f(z,w)\not\equiv 0$, the associated PLSI polynomial is zero-free in $\overline{\D^2}$.  This conjecture became known in the engineering community as \textit{Shanks' Conjecture}.

In 1975, Genin and Kamp \cite{GK} gave a counterexample to Shanks' Conjecture.  In particular, they discovered a polynomial with zeros in $\D^2$  such that the associated PLSI polynomial has zeros in $\D^2$. A few years later, they developed a method \cite{GK1} to construct polynomials with which the associated PLSI polynomial is guaranteed to have zeros in $\D^2$.  At this point, it was unclear whether it was possible to impose conditions on the polynomial $f(z,w)$ to guarantee that the associated PLSI polynomial is zero-free in $\overline{\D^2}$.  In 1980, Delsarte, Genin, and Kamp \cite{DGK} made the following conjecture:  for any polynomial $f(z,w)$ that doesn't vanish in $\overline{\D^2}$, the associated PLSI polynomial is zero-free in $\overline{\D^2}$.  Although there have been several attempts to prove this conjecture, it is still unresolved.  This conjecture is referred to as \textit{Weak Shanks' Conjecture}.

Since our primary motivation for writing this paper is to understand the zeros of OPAs in the context of $H^p$, it's natural for us to wonder if ``Shanks-type'' results hold in this setting.  In fact, we already know that such a result holds in $H^2$.  More specifically, we know that the OPA $q_{n,2}[f,1]$ is zero-free in $\overline{\D}$ whenever $f\in H^2$ and $f(0)\neq 0$.  With this in mind, we can start by asking if the OPA $q_{n,p}[f,1]$ is zero-free in $\overline{\D}$ whenever $f\in H^p$ and $f(0)\neq 0$.  Moreover, we can ask if it's necessary to impose the condition $f(0)\neq 0$.  On a different note, if there are values of $n$ for which the OPAs have a zero in $\D$, then we should ask if there's a limit to how far in $\D$ the zeros can lie.  To begin our investigation, we establish the following proposition.

            \begin{prop}\label{propzero5}
Let $1<p<\infty$, $f\in H^p\setminus\{0\}$, and $n\in\N$.  If there exists some $z_0\in\D$ such that $q_{n,p}[f,1](z_0)=0$, then
\begin{equation*}
    \sqrt{1-\|q_{n,p}[f,1]f-1\|_p^p}\leq |z_0|.
\end{equation*}
\end{prop}

\begin{proof}
Let $Q:=q_{n,p}[f,1]$.  Suppose there exists some $z_0\in\D$ such that $Q(z_0)=0$.  Define the analytic functions
\begin{equation*}
    \varphi(z):=\frac{z_0-z}{1-\overline{z_0}z}\quad \text{and}\quad \psi(z):=\bigg\{\frac{1-|z_0|^2}{(1-\overline{z_0}z)^2}\bigg\}^{\frac{1}{p}}.
\end{equation*}
Let $\psi C_{\varphi}:H^p\rightarrow H^p$ denote the weighted composition operator $h(z)\mapsto\psi(z)h\big(\varphi(z)\big)$.  Since $\big|\psi C_{\varphi}(Qf-1)\big|^p$ is subharmonic in $\D$, it follows that 
\begin{align*}
    1-|z_0|^2&=\big|\psi C_{\varphi}(Qf-1)(0)\big|^p\\
    &\leq \frac{1}{2\pi i}\int_{\T}\Big|Q\Big(\frac{z_0-z}{1-\overline{z_0}z}\Big)f\Big(\frac{z_0-z}{1-\overline{z_0}z}\Big)-1\Big|^p\frac{1-|z_0|^2}{|1-\overline{z_0}z|^2}\frac{dz}{z}\\
    &=\frac{1}{2\pi i}\int_{\T}|Q(w)f(w)-1|^p\frac{dw}{w}\\
    &=\|Qf-1\|_p^p.
\end{align*}
Equivalently, we have that $\sqrt{1-\|Qf-1\|_p^p}\leq |z_0|$.
\end{proof}

Proposition \ref{propzero5} shows that if the OPA has a zero in $\D$ and the quantity $\|q_{n,p}[f,1]f-1\|_p$ is small, then we can expect the zero to be close to $\T$.  Now, it's still interesting to know which functions $f\in H^p$ correspond to a quantity $\|q_{n,p}[f,1]f-1\|_p$ that's close to $1$.  Of course, this quantity is equal to $1$ whenever $q_{n,p}[f,1]\equiv 0$.  In the case where $p=2$, it is well-known that $q_{n,p}[f,1]\equiv 0$ if and only if $f(0)=0$ (see, e.g., \cite{BKLSS}).  Interestingly enough, the following proposition shows that this is true for all $1<p<\infty$.

                      \begin{prop}\label{propzero}
Let $1<p<\infty$, $f\in H^p\setminus\{0\}$, and $n\in\N$.  The following are equivalent:
\begin{enumerate}[(i)]
    \item $f(0)=0$
    \item $q_{n,p}[f,1]\equiv 0$
    \item $q_{n,p}[f,1](0)=0$
\end{enumerate}
\end{prop}

\begin{proof}
Let $q_{n,p}:=q_{n,p}[f,1]$.  We first suppose that $f(0)=0$.  Since $|q_{n,p}f-1|^p$ is subharmonic in $\D$, it follows that 
\begin{align*}
    1&=|q_{n,p}(0)f(0)-1|^p\\
    &\leq\frac{1}{2\pi}\int_{-\pi}^{\pi}|q_{n,p}f-1|^pdt\\
    &=\|q_{n,p}f-1\|_p^p.
\end{align*}
Since $\|q_{n,p}f-1\|_p\leq 1$, we have that $\|q_{n,p}f-1\|_p=1$.  By uniqueness of OPAs, it follows that $q_{n,p}\equiv 0$.  This shows that $(i)$ implies $(ii)$.

To prove that $(ii)$ implies $(iii)$ is trivial.  Now, suppose that $q_{n,p}(0)=0$.  In a similar argument as above, it follows that $q_{n,p}\equiv 0$.  Therefore, Proposition \ref{orthop} implies that
\begin{equation}\label{harm}
    \frac{1}{2\pi}\int_{-\pi}^{\pi}\overline{f(e^{it})}dt=0.
\end{equation}
By conjugating both sides of (\ref{harm}), we see that 
\begin{equation*}
    f(0)=\frac{1}{2\pi}\int_{-\pi}^{\pi}f(e^{it})dt=0.
\end{equation*}
This shows that $(iii)$ implies $(i)$.
\end{proof}

As a result of the previous two propositions, we are able to address the question about how far in $\D$ the zeros of OPAs can lie.

\begin{thm}\label{main}
Let $1<p<\infty$, $f\in H^p$, $f(0)\neq 0$, and $n\in\N$.  Then $q_{n,p}[f,1](z)$ is zero-free for $|z|<\sqrt{1-\|q_{0,p}[f,1]f-1\|_p^p}$. 
\end{thm}

\begin{proof}
Note that since $f(0)\neq 0$, it follows from Proposition \ref{propzero} that $\|q_{0,p}[f,1]f-1\|_p<1$, i.e., $\big\{z\in\C\colon |z|<\sqrt{1-\|q_{0,p}[f,1]f-1\|_p^p}\big\}$ is nonempty.  Now, suppose $q_{n,p}[f,1](z_0)=0$ for some $z_0\in\D$.  As a result of Proposition \ref{propzero5} and the fact that $\big\{\|q_{k,p}[f,1]f-1\|_p\big\}_{k=0}^{\infty}$ is a decreasing sequence of nonnegative numbers, we conclude that
\begin{align*}
    \sqrt{1-\|q_{0,p}[f,1]f-1\|_p^p}&=\sqrt{1-\sup_{k\in\N}\|q_{k,p}[f,1]f-1\|_p^p}\\
    &=\inf_{k\in\N}\sqrt{1-\|q_{k,p}[f,1]f-1\|_p^p}\\
    &\leq \sqrt{1-\|q_{n,p}[f,1]f-1\|_p^p}\\
    &\leq |z_0|.
\end{align*}
Therefore, $q_{n,p}[f,1](z)$ is zero-free for $|z|<\sqrt{1-\|q_{0,p}[f,1]f-1\|_p^p}$.
\end{proof}

For a function $f\in H^p$ with $f(0)\neq 0$, Theorem \ref{main} gives us considerable insight into the zeros of $q_{n,p}[f,1]$.  In fact, it provides us with the radius of a disk in which $q_{n,p}[f,1]$ is zero-free for every $n\in\N$.  Now, the existence of this disk is partially due to Proposition \ref{propzero}:  since $f(0)\neq 0$, we have that $q_{0,p}[f,1]\not\equiv 0$ and hence $\sqrt{1-\|q_{0,p}[f,1]f-1\|_p^p}>0$.  Nonetheless, the following result (which was noted in \cite{BKLSS} for a particular class of weighted Hilbert spaces of analytic functions) demonstrates another interesting consequence of Proposition \ref{propzero}. 
%For a function $f\in H^p$ with $f(0)\neq 0$, Theorem \ref{main} gives us considerable insight into the zeros of $q_{n,p}[f,1]$.  In fact, it provides us with the radius of a disk in which $q_{n,p}[f,1]$ is zero-free for every $n\in\N$.  Now, the existence of this disk is partially due to Proposition \ref{propzero}:  If there happened to be a function $f\in H^p$ such that $f(0)\neq 0$ and $q_{0,p}[f,1]\equiv 0$, then the set $\big\{z\in\C\colon |z|<\sqrt{1-\|q_{0,p}[f,1]f-1\|_p^p}\big\}$ would be empty.  Nonetheless, the following result demonstrates another interesting consequence of Proposition \ref{propzero}. 

                 \begin{prop}\label{lem}
Let $1<p<\infty$.  For each $n\in\N\setminus\{0\}$, let 
\begin{equation*}
    M_n:=\inf \{|z|\colon q_{n,p}[f,1](z)=0\ \text{for at least one}\ f\in H^p\ \text{with}\ f(0)\neq 0\}.
\end{equation*}
Then $M_1\leq M_n$.
\end{prop}
\begin{proof}
For each $n\in\N\setminus\{0\}$, define the set
\begin{equation*}
    E_n:=\{|z|\colon q_{n,p}[f,1](z)=0\ \text{for at least one}\ f\in H^p\ \text{with}\ f(0)\neq 0\}.
\end{equation*}
If $|z_0|\in E_n$, then there exists some $f\in H^p$, with $f(0)\neq 0$, such that \\$q_{n,p}[f,1](z_0) = 0$.  Define the function
\begin{equation*}
    h(z):=\frac{q_{n,p}[f,1](z)f(z)}{z-z_0},
\end{equation*}
and note that 
\begin{equation}
    \|q_{n,p}[f,1]f-1\|_{p}=\|(z-z_0)h-1\|_{p}.\label{eqlem}
\end{equation}
Now, suppose there exists some $q\in\mathcal{P}_1$ such that 
\begin{equation*}
    \|(z-z_0)h-1\|_{p}>\|qh-1\|_{p}.
\end{equation*}
From (\ref{eqlem}), it follows that 
\begin{equation*}
    \|q_{n,p}[f,1]f-1\|_{p}>\|qh-1\|_{p}=\bigg\|q\frac{q_{n,p}[f,1]}{z-z_0}f-1\bigg\|_{p},
\end{equation*}
a contradiction since $q\frac{q_{n,p}[f,1]}{z-z_0}$ is a polynomial of degree at most $n$.  We conclude that $q_{1,p}[h,1](z)=z-z_0$.  Now, since $f(0)\neq 0$, it follows from Proposition \ref{propzero} that $q_{n,p}[f,1](0)\neq 0$.  This implies that $h(0)\neq 0$.  We then see that $|z_0|\in E_1$, and hence $E_n\subset E_1$.  We conclude that $M_1\leq M_n$.
\end{proof}

In an effort to find a lower bound for the zeros of $q_{n,p}[f,1]$, where $f\in H^p$ and $f(0)\neq 0$, Proposition \ref{lem} suggests that we should focus our attention to the case where $n=1$.  Now, by making the additional assumption that $2\leq p<\infty$, we can gain insight into the zeros of all the OPAs by implementing the test from Theorem \ref{condition}.  With this in mind, we pose the following question.

      \begin{OQ}\label{oq3}
Let $2\leq p<\infty$, $f\in H^p$, $f(0)\neq 0$, $g:=\{q_{1,p}[f,1]f-1\}^{\frac{p-2}{2}}$, and $g\not\equiv 0$.  Is it true that $|\langle g,fg\rangle |>|\langle g, zfg\rangle |$?
\end{OQ}

Note that if $g\equiv 0$, then it follows from Remark \ref{remark} that $q_{1,p}[f,1]$ is zero-free in $\overline{\D}$.  Otherwise, it follows from Theorems \ref{condition} and \ref{cor4} that the OPA $q_{1,p}[f,1]$ is zero-free in $\overline{\D}$ if and only if $|\langle g,fg\rangle|>|\langle g, zfg\rangle|$.  Therefore, if the answer to Open Question \ref{oq3} is ``Yes'', then we could deduce from Propositions \ref{propzero} and \ref{lem} that $q_{n,p}[f,1]$ is zero-free in $\overline{\D}$ for every $n\in\N$ and every $f\in H^p$ with $f(0)\neq 0$.

Although many of our discussions have been focusing on functions in $H^p$, there are other functions in $L^p$ that are worthy of our attention, e.g., the modulus of any nontrivial function in $H^p$.  In the following proposition, we use Theorem \ref{condition} to gain insight into the zeros of OPAs associated with some of these functions.

                    \begin{prop}\label{abs}
Let $2\leq p<\infty$.  If $f$ is a real-valued nonnegative function in $L^p\setminus\{0\}$, then $q_{1,p}[f,1]$ is zero-free in $\overline{\D}$.
\end{prop}

\begin{proof}
Let $g:=\{q_{1,p}[f,1]f-1\}^{\frac{p-2}{2}}$.  If $g\equiv 0$, then $q_{1,p}[f,1]$ must be a nonzero constant, i.e., $q_{1,p}[f,1]$ is zero-free in $\overline{\D}$.  If $g\not\equiv 0$, then it follows from Theorem \ref{cor4} that $q_{1,p}[f,1]=q_{1,2}[fg,g]$.  Since
\begin{align*}
    |\langle g,zfg\rangle|&=\big|\big\langle g,zf^{\frac{1}{2}}f^{\frac{1}{2}}g\big\rangle\big|\\
    &=\big|\big\langle gf^{\frac{1}{2}},zf^{\frac{1}{2}}g\big\rangle\big|\\
    &<\big\|gf^{\frac{1}{2}}\big\|_2^2\\
    &=\big|\big\langle 1,|g|^2f\big\rangle\big|\\
    &=|\langle g,fg\rangle|,
\end{align*}
we conclude from Theorem \ref{condition} that $q_{1,p}[f,1]$ is zero-free in $\overline{\D}$.
\end{proof}

\begin{ex}
Let $f(z)=|1-z|^2$.  From Proposition \ref{abs}, it follows that $q_{1,2}[f,1]$ is zero-free in $\overline{\D}$.  More specifically, since $f(z)=-z+2-z^{-1}$ on $\T$, we can easily show that $q_{1,2}[f,1](z)=\frac{1}{10}z+\frac{2}{5}$.
\end{ex}

As of this point in the paper, our computations of OPAs have only pertained to the case where $p=2$.  In the following section, we develop a method that allows us to compute the coefficients of OPAs for many different values of $p$; this method is based on the orthogonality condition that was presented in Section \ref{integralcharacterization}.

%%%%%%%%%%%%%%%%%%%%%%%%%%%%%%%%%%%%%%%%%%%%%%%%%%%%%%%%%%%%%%%
\section{Computation of Coefficients}\label{comp}

As we saw in Section \ref{sec3}, the coefficients of OPAs in $L^2$ can be computed by using Proposition \ref{ortho}.  In fact, these coefficients can be expressed as the solution to a system of linear equations for which the associated matrix can easily be inverted.  Now, recall that Proposition \ref{orthop} is a generalization of Proposition \ref{ortho} to the $L^p$ setting.  Therefore, one would hope that this result could be used to compute the coefficients of OPAs in $L^p$; after making restrictions on the functions and considering specific values of $p$, we find that we are able to solve for the coefficients numerically; in light of Proposition \ref{lem}, we will focus most of our computations on the case where $n=1$.

%However, for values of $p$ other than $2$, we see that the coefficients are expressed as the solution to a system of nonlinear equations -- a rather complicated one, indeed!  

Before we begin the numerical computations, let's consider an example in which $n=0$.

\begin{ex}
Let $f(z)=1-z$.  Then $q_{0,p}[f,1]\equiv\frac{1}{2}$ for all integers $p>1$.
\end{ex}     

\begin{proof}
Let $Q:=\frac{1}{2}$.  Note that 
\begin{align}\label{eq9}
    \frac{1}{2\pi}&\int_{-\pi}^{\pi}|Qf-1|^{p-1}\sgn(Qf-1)e^{-ikt}\overline{f}dt\nonumber\\
    &=\frac{1}{2\pi}\int_{-\pi}^{\pi}|Qf-1|^{p-2}(Qf-1)\overline{f}dt\nonumber\\
    &=\frac{1}{2\pi i}\int_{\T}\bigg|-\frac{1}{2}z-\frac{1}{2}\bigg|^{p-2}\bigg(-\frac{1}{2}z-\frac{1}{2}\bigg)(1-\overline{z})\frac{dz}{z} \nonumber\\
    &=\frac{1}{2\pi i}\int_{\T}\bigg(-\frac{1}{2}z-\frac{1}{2}\bigg)^{\frac{p-2}{2}}\bigg(-\frac{1}{2}\overline{z}-\frac{1}{2}\bigg)^{\frac{p-2}{2}}\bigg(-\frac{1}{2}z-\frac{1}{2}\bigg)(1-\overline{z})\frac{dz}{z}\nonumber\\ 
    &=\frac{1}{2\pi i}\int_{\T}\bigg(-\frac{1}{2}z-\frac{1}{2}\bigg)^{p-1}\frac{z-1}{z^{\frac{p}{2}}}\frac{dz}{z}\nonumber\\
    &=\bigg(-\frac{1}{2}\bigg)^{p-1}\frac{1}{2\pi i}\int_{\T}\Bigg[\sum_{j=0}^{p-1}\binom{p-1}{j}z^j\Bigg]\frac{z-1}{z^{\frac{p}{2}}}\frac{dz}{z}\nonumber\\
    &=\bigg(-\frac{1}{2}\bigg)^{p-1}\frac{1}{2\pi i}\sum_{j=0}^{p-1}\binom{p-1}{j}\bigg[\int_{\T}z^{j-(\frac{p-2}{2})}\frac{dz}{z}-\int_{\T}z^{j-\frac{p}{2}}\frac{dz}{z}\bigg].
\end{align}
If $p$ is odd, then each integral in (\ref{eq9}) is zero for every $j$.  In this case,
\begin{equation*}
    \frac{1}{2\pi}\int_{-\pi}^{\pi}|Qf-1|^{p-1}\sgn(Qf-1)e^{-ikt}\overline{f}dt=0. 
\end{equation*}
If $p$ is even, then the first integral in (\ref{eq9}) is nonzero only when $j=\frac{p-2}{2}$.  The second integral is nonzero only when $j=\frac{p}{2}$.  Since
\begin{equation*}
    \binom{p-1}{\frac{p-2}{2}}=\frac{(p-1)!}{\big(\frac{p}{2}\big)!\,\big(\frac{p-2}{2}\big)!}=\binom{p-1}{\frac{p}{2}},
\end{equation*}
it follows that 
\begin{equation*}
   \frac{1}{2\pi}\int_{-\pi}^{\pi}|Qf-1|^{p-1}\sgn(Qf-1)e^{-ikt}\overline{f}dt=0.
\end{equation*}
In either case, we conclude from Proposition \ref{orthop} that 
\begin{equation*}
 Q=q_{0,p}[f,1]\equiv\frac{1}{2}.   
\end{equation*}

\end{proof}
           
%We have shown that if $f(z)=1-z$ and $p$ is an integer greater than $1$, then $\frac{1}{2}$ minimizes the norm $\|af-1\|_p$, where $a$ varies over $\C$.  In Proposition \ref{uniquenesseg}, we showed that this is also true for $p=\infty$.  Now, it would be interesting to know if this is the case for all $1\leq p\leq\infty$.  As a general question, if $f$ is a given polynomial, is it true that the OPAs $q_{0,p}[f,1]$ are the same value for each $1<p<\infty$?  If so, does this value minimize the norm $\|af-1\|_p$ for $p=1$ and $p=\infty$, where $a$ varies over $\C$?

Let us now begin to formulate a general setting to carry out our computations; for simplicity, we will only consider OPAs of the form $q_{n,p}[f,1]$.  We will see that it works to our advantage to make certain assumptions on the function $f$.  One of our assumptions will be based on the observation that conjugating a function's Fourier coefficients does not change its norm.  More specifically, if $f$ belongs to $L^p$ and has Fourier series $\sum_{k=-\infty}^{\infty}c_kz^k$, then the function $\Tilde{f}$ with Fourier series $\sum_{k=-\infty}^{\infty}\overline{c}_kz^k$ belongs to $L^p$ and satisfies $\|\Tilde{f}\|_p=\|f\|_p$.  This allows us to deduce the following result.

                \begin{prop}\label{real}
Let $1<p<\infty$, $f\in L^p\setminus\{0\}$, and $n\in\N$.  If the Fourier coefficients of $f$ are all real, then the coefficients of $q_{n,p}[f,1]$ are all real.
\end{prop}
\begin{proof}
Let $q_{n,p}:=q_{n,p}[f,1]$. Then we see that 
\begin{align*}
    \inf_{q\in\Pn}\|qf-1\|_p&=\|q_{n,p}f-1\|_p\\
    &=\|(q_{n,p}f-1)^{\sim}\|_p\\
    &=\|\Tilde{q}_{n,p}f-1\|_p.
\end{align*}
By uniqueness of OPAs, we have that $q_{n,p}=\Tilde{q}_{n,p}$.
\end{proof}

In the computations that later follow, we will make the assumption that the Fourier coefficients of the function $f$ are all real.  Under this assumption, Proposition \ref{real} will drastically simplify the equations involved in solving for the OPA coefficients.  Nonetheless, the following theorem gives us a method for computing the coefficients of OPAs  whenever $f$ is a polynomial and $p$ is even.  

                \begin{theorem}\label{thmpoly}
Let $n,N\in\N$ and $Q\in\Pn$.  Suppose that $f$ is a polynomial of degree at most $N$ and $p\geq 2$ is even.  Define the polynomials $G(z):=Q(z)f(z)-1$, $R(z):=z^{n+N}\overline{G\big(\frac{1}{\overline{z}}\big)}$, and $P(z):=z^NG(z)^{\frac{p}{2}}R(z)^{\frac{p-2}{2}}\overline{f\big(\frac{1}{\overline{z}}\big)}$.  Then
\begin{equation*}
   \left.\frac{d^{k}}{dz^{k}}P(z)\right|_{z=0}=0
\end{equation*}
for $\frac{p}{2}(n+N)-n\leq k\leq\frac{p}{2}(n+N)$ if and only if $Q=q_{n,p}[f,1]$.
\end{theorem}

\begin{proof}
Note that 
\begin{align*}
    \frac{1}{2\pi}\int_{-\pi}^{\pi}|Qf-1|&^{p-1}\sgn(Qf-1)e^{-ikt}\overline{f}dt\\
    &=\frac{1}{2\pi}\int_{-\pi}^{\pi}|Qf-1|^{p-2}(Qf-1)e^{-ikt}\overline{f}dt\\&=\frac{1}{2\pi i}\int_{\T}\frac{(Qf-1)^{\frac{p}{2}}(\overline{Qf}-1)^{\frac{p-2}{2}}\overline{f}}{z^k}\frac{dz}{z}\\
    &=\frac{1}{2\pi i}\int_{\T}\frac{z^NG(z)^{\frac{p}{2}}R(z)^{\frac{p-2}{2}}\overline{f\big(\frac{1}{\overline{z}}\big)}}{z^{k+\frac{p}{2}(n+N)-n}}\frac{dz}{z}\\
    &=\frac{1}{2\pi i}\int_{\T}\frac{P(z)}{z^{k+\frac{p}{2}(n+N)-n}}\frac{dz}{z}.
\end{align*}
Therefore, it follows from Proposition \ref{orthop} that
\begin{equation}
   \frac{1}{2\pi i}\int_{\T}\frac{P(z)}{z^{k+\frac{p}{2}(n+N)-n}}\frac{dz}{z}=0\quad\text{for}\quad k=0,\dots,n \label{cauchy}
\end{equation}
if and only if $Q=q_{n,p}[f,1]$.  Since $P$ is a polynomial, we can easily see that (\ref{cauchy}) is equivalent to 
\begin{equation}\label{dereq}
   \left.\frac{d^{k}}{dz^{k}}P(z)\right|_{z=0}=0\quad \text{for}\quad \frac{p}{2}(n+N)-n\leq k\leq\frac{p}{2}(n+N).
\end{equation}
Hence, (\ref{dereq}) is true if and only if $Q=q_{n,p}[f,1]$.
\end{proof}

Let's have a look at an example to demonstrate this method.

%As indicated in our previous discussion, we will only consider the case where the coefficients of the polynomial $f$ are all real.  

                \begin{ex}\label{ex6.6}
Let $f(z)=1-z$; let $z_0$ be the zero of $q_{1,p}[f,1]$.
\begin{enumerate}[(i)]
    \item $q_{1,2}[f,1](z)\approx 0.3333333333z+0.6666666667,\quad z_0\approx -2.00000$\label{one}
    \item $q_{1,4}[f,1](z)\approx 0.3734388420z+0.6265611579, \quad z_0\approx -1.67781$\label{two}  
    \item $q_{1,6}[f,1](z)\approx 0.3964823122z+0.6035176878, \quad z_0\approx -1.52218 $\label{three}
    \item $q_{1,8}[f,1](z)\approx 0.4117075962z+0.5882924038, \quad z_0\approx -1.42891 $\label{four}
    \item $q_{1,10}[f,1](z)\approx 0.4226290585z+0.5773709415, \quad z_0\approx -1.36614 $\label{five}
    \item $q_{1,20}[f,1](z)\approx 0.4336619002z+0.5313377953, \quad z_0\approx -1.22524 $\label{six}
    \item $q_{1,30}[f,1](z)\approx 0.2117474393z+0.2701449218, \quad z_0\approx -1.27578 $\label{seven}
\end{enumerate}
\end{ex}

%\begin{center}
 %    \includegraphics[width=10cm]{Zeros_1.jpg}
%\end{center}

\begin{proof}
We only show the work for (\ref{two}).  However, note the similarity between the OPA in (\ref{one}) and the OPA in Example \ref{ex3.4}.  Let $Q:=q_{1,4}[f,1]$.  Define the polynomials $G(z)$, $R(z)$, and $P(z)$ as in Theorem \ref{thmpoly}.  By setting $q_{1,4}[f,1]=a_1z+a_0$, it is easy to see that
\begin{equation*}
    G(z)=-a_1z^2+(a_1-a_0)z+(a_0-1).
\end{equation*}
Moreover, we have that
\begin{equation*}
    z\overline{f\bigg(\frac{1}{\overline{z}}\bigg)}=z-1.
\end{equation*}
Now, it follows from Proposition \ref{real} that $a_0$ and $a_1$ are real.  Therefore,
\begin{align*}
    R(z)&=z^2\overline{G\bigg(\frac{1}{\overline{z}}\bigg)}\\
    &=z^2G\bigg(\frac{1}{z}\bigg)\\
    &=(a_0-1)z^2+(a_1-a_0)z-a_1.
\end{align*}
We thus find that 
\begin{equation*}
    P(z)=\big\{-a_1z^2+(a_1-a_0)z+(a_0-1)\big\}^2\big\{(a_0-1)z^2+(a_1-a_0)z-a_1\big\}(z-1).
\end{equation*}

To solve for $a_0$ and $a_1$, we set 
\begin{equation*}
   \left.\frac{d^{k}}{dz^{k}}P(z)\right|_{z=0}=0  
\end{equation*}
for $k=3$ and $k=4$.  This gives us the nonlinear equations

\begin{align*}
  -24a_1^3+78a_1^2a_0&-72a_1a_0^2+36a_0^3\\
  &-36a_1^2+60a_1a_0-54a_0^2-12a_1+30a_0-6=0  
\end{align*}
and
\begin{align*}
  144a_1^3-288a_1^2a_0&+312a_1a_0^2-96a_0^3\\
  &+72a_1^2-288a_1a_0+120a_0^2+96a_1-48a_0=0.  
\end{align*}
Through the use of numerical methods, we find that $a_1\approx 0.3734388420$ and $a_0\approx 0.6265611579$.

\end{proof}

\begin{center}
     \includegraphics[width=10cm]{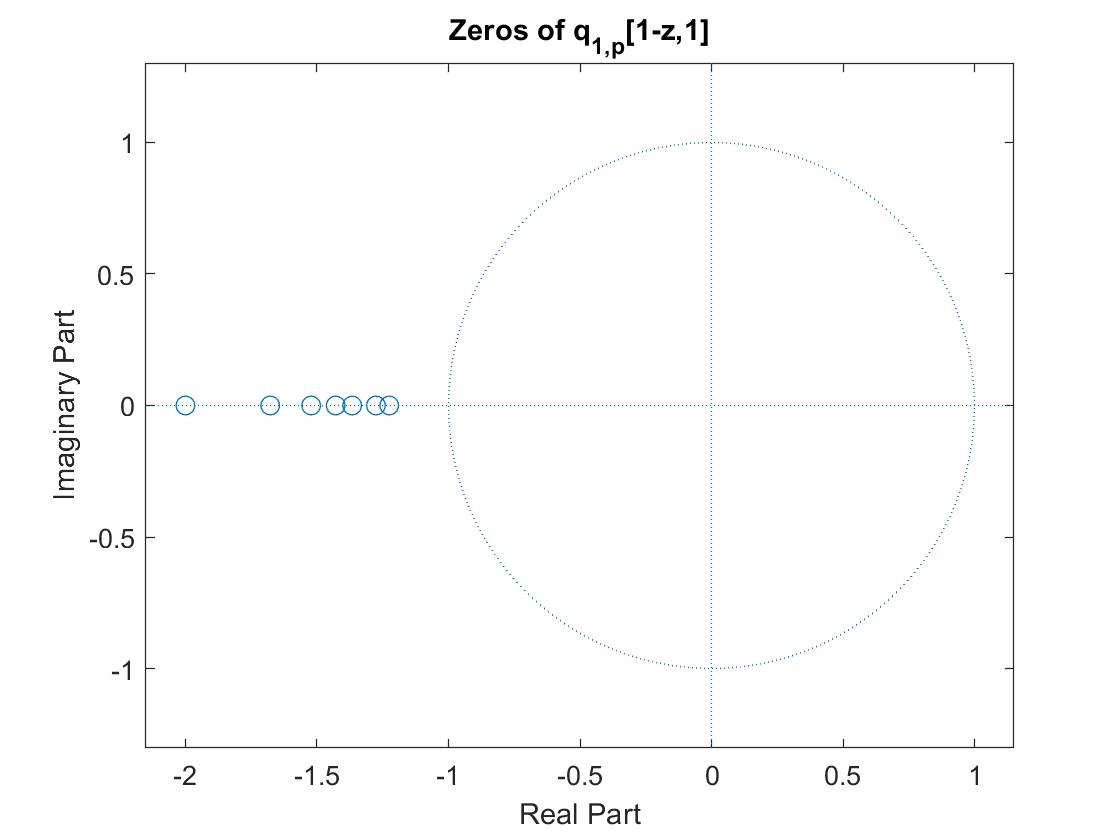}
\end{center}

In the above example, it seems plausible that the zeros of $q_{1,p}[f,1]$ remain outside of $\overline{\D}$ for all $p$.  Moreover, the zeros seem to converge as $p\rightarrow\infty$.  As a general question, for a polynomial $f$ with $f(0)\neq 0$, do the zeros of $q_{1,p}[f,1]$ converge as $p\rightarrow\infty$?  Under these conditions, it would be interesting to determine the smallest disk that contains all the zeros.  Perhaps this disk is disjoint from the closed unit disk $\overline{\D}$.

%%%%%%%%%%%%%%%%%%%%%%%%%%%%%%%%%%%%%%%%%%%%%%%%%%%%%%%%%
\section{Concluding Remarks}

The polynomials obtained by least-squares approximation have been a topic of interest to engineers and mathematicians over the last several decades.  In the engineering community, many have studied these polynomials in reference to the Hardy spaces $H^2(\D^2)$ and $H^2(\D)$.  In the mathematics community, many have studied these polynomials in reference to the Dirichlet-type spaces $\mathcal{D}_{\alpha}$, $\alpha\in\R$.  In this paper, we introduced an analogous version of these polynomials in the space $L^p$, $1\leq p\leq\infty$.  Our main motivation for doing so was to try and extend ``Shanks-type'' results to a larger collection of functions.  As a result of our many discussions, we are convinced that such a result holds for functions in $H^p$.  Accordingly, we make the following conjecture.

\begin{conj}\label{shanksconj}
Let $1<p<\infty$, $f\in H^p$, and $n\in\N$.  If $f(0)\neq 0$, then $q_{n,p}[f,1]$ is zero-free in $\overline{\D}$.
\end{conj}

For the case where $2\leq p<\infty$, we can approach this conjecture through the lens of the Hilbert space $L^2$; this is because we can express the OPA $q_{n,p}[f,1]$ in terms of an OPA in $L^2$.  In particular, one could formulate the theory through the use of Moore-Penrose inverses (see \cite{Iz} for a discussion in the $H^2$ setting).  More specifically, for any $f\in L^{\infty}\setminus\{0\}$, $g\in L^2$, and $n\in\N$, one could write $q_{n,2}[f,g]=(M_fE_n)^{\dagger}(g)$, where $M_f:L^2\rightarrow L^2$ is the multiplication operator defined by $h\mapsto fh$, $E_n:L^2\rightarrow\Pn$ is the orthogonal projection of $L^2$ onto $\Pn$, and $(M_fE_n)^{\dagger}$ is the Moore-Penrose inverse of $M_fE_n$.  Accordingly, one could take an operator theoretic approach toward gaining insight into the OPA.

\vspace{5mm}

\noindent
\textbf{Acknowledgments.}  I would like to thank Catherine B\'en\'eteau for the many enlightening conversations that we had about these topics.  Additionally, I would like to thank her, Dmitry Khavinson, and the referees for reviewing the draft of this paper and providing valuable feedback.  

%\vspace{5mm}

%\noindent
%\textbf{Conflict of Interest.}  The author declares that there is no conflict of interest in connection with this manuscript.
%\textit{The author declares no conflict of interest in connection with the publication of this article.}

%%%%%%%%%%%%%%%%%%%%%%%%%%%%%%%%%%%%%%%%%%%%%%%%%%%%%%%%%%%%%%

\end{document}